\documentclass[12pt]{amsart}
\usepackage{bbm}
\usepackage{amscd, amssymb, dsfont, upgreek, mathrsfs}
\input{xy}
\xyoption{all}
\newtheorem{Thm}{Theorem}[section]
\newtheorem{Conj}[Thm]{Conjecture}
\newtheorem{Prop}[Thm]{Proposition}

\newtheorem{Def/Thm}[Thm]{Definition/Theorem}
\newtheorem{Cor}[Thm]{Corollary}

\theoremstyle{remark}
\newtheorem{Rmk}[Thm]{Remark}

\numberwithin{equation}{section}
\newcommand{\ti }{\times}
\newcommand{\ot }{\otimes}
\newcommand{\ra }{\rightarrow}

\newcommand{\Hom }{{\mathrm{Hom}}}

\newcommand{\Pic}{{\mathrm{Pic}}}

\newcommand{\G}{{\bf G}}

\newcommand{\PP }{{\mathbb P}}

\newcommand{\QQ }{{\mathbb Q}}
\newcommand{\CC }{{\mathbb C}}
\newcommand{\ZZ }{{\mathbb Z}}

\newcommand{\ke }{{\varepsilon }}

\newcommand{\kl }{{\lambda}}

\newcommand{\vir}{\mathrm{vir}}

\newcommand{\WmodG}{W/\!\!/\G}

\newcommand{\T}{{\bf T}}

\newcommand{\re}{\mathrm{e}}

\newcommand{\lan}{\langle}
\newcommand{\ran}{\rangle}
\newcommand{\lla}{\langle\!\langle}
\newcommand{\rra}{\rangle\!\rangle}

\newcommand{\WtG}{W/\!\!/_{\!\theta}\G}
\newcommand{\VtG}{V/\!\!/_{\!\theta}\G}

\newcommand{\bx}{\mathbf{x}}

\newcommand{\ret}{\re ^{\T}}
\newcommand{\rV}{\mathrm{V}}

\newcommand{\uuI}{{\underline{I}}}

\begin{document}

\title[Mirror Theorem for Elliptic Quasimap Invariants]{Mirror Theorem for Elliptic Quasimap Invariants}

\begin{abstract} 
We propose and prove a mirror theorem for the elliptic quasimap invariants 
of smooth Calabi-Yau complete intersections in projective spaces. 
The theorem combined with the wall-crossing 
formula appeared in \cite{CKg} implies mirror theorems of Zinger and  Popa for the elliptic 
Gromov-Witten invariants of those varieties.  This paper and the wall-crossing formula 
provide a unified framework for the mirror theory of rational and elliptic Gromov-Witten invariants.
\end{abstract}

\author{Bumsig Kim}
\address{School of Mathematics, Korea Institute for Advanced Study,
85 Hoegiro, Dongdaemun-gu, Seoul, 02455, Korea}
\email{bumsig@kias.re.kr}

\author{Hyenho Lho}
\address{School of Mathematics, Korea Institute for Advanced Study,
85 Hoegiro, Dongdaemun-gu, Seoul, 02455, Korea}
\email{hyenho@kias.re.kr}

%\date{\today}

\maketitle

%\tableofcontents

\section{Introduction}

Let $W$ be a codimension $r$ affine subvariety in $\CC ^{n}$ defined by homogeneous degree $l_1, ..., l_r$ 
polynomials such that the origin is the only singular point of $W$. 
Assume $$\sum_{a=1}^r l_a = n$$ and let $\G := \CC^*$ act on $\CC ^n$ by the standard diagonal action  so that 
its associated GIT quotient $$X :=\WmodG$$ is a codimension $r$, nonsingular Calabi-Yau complete intersection in $\PP ^{n-1}$.

With this setup, for each positive rational number $\ke$ there are so-called $\ke$-stable quasimap moduli space
$$Q^{\ke}_{g, 0}(X, d)$$ with the canonical virtual fundamental class $[Q^{\ke}_{g, 0}(X, d)]^{\vir}$ (see \cite{CKM}).
We are mainly interested in the space $Q^{\ke}_{g, 0}(X, d)$ with small enough $\ke$ with respect to 
degree $d$, which will be denoted by $Q^{0+}_{g, 0}(X, d)$ and also simply by $Q_{g, 0}(X, d)$.
When $\ke >2$, $Q^\ke _{g, 0}(X, d)$ coincides with the moduli space 
$\overline{M}_{g, 0}(X, d)$ of stable maps, which will be denoted also by
$Q^{\infty}_{g, 0}(X, d)$.

When $g=1$, the virtual dimension of $Q^{\ke}_{g, 0}(X, d)$ is always zero.
The main goal of this paper is to discover an explicit description of 
$\deg [Q_{1,0}(X, d)]^{\vir}$ in terms of the Givental's $I$-function for $X$. 
Let \begin{align*}  \lan\; \ran ^{\ke }_{1,0} := \sum_{d=1}^{\infty} q^d \deg [Q^{\ke}_{1,0}(X, d)]^{\vir} , \end{align*}
where $q$ is a formal Novikov variable.
We express the generating function $\lan\; \ran ^{\ke }_{1,0} $
 in terms of Givental's $\T$-equivariant $I$-function  for $X$,
where $\T:=(\CC^*)^n$ is the complex torus group acting on $\PP ^{n-1}$ (see \cite{Gequiv}).

The equivariant $I$-function is the $H^*_{\T}(\PP ^{n-1})\ot \QQ(\lambda)$-valued formal function in formal variables $q, z, t_H$:
\begin{align}\label{equi_I} I_{\T}(t,  q):= e^{t_HH/z}\sum _{d=0}^{\infty} q^d e^{t_Hd}
 \frac{\prod _{a=1}^r  \prod _{k=1}^{l_ad} (l_aH + kz)}{\prod _{k=1}^{d} \prod _{j=1}^n (H-\lambda _j +kz)} ,
\end{align}
where $\lambda _1, ..., \lambda _n$ are the $\T$-equivariant parameters; $\QQ (\lambda)$ denotes
the quotient field of the polynomial ring in $\lambda _1, ...,\lambda _n$; $H$ is 
the $\T$-equivariant hyperplane class; and $t :=t_H H$.

Let $\lambda _0$ be another formal parameter.
Consider the restriction $I_{\T}(0, q)|_{p_i}  $ of $I_{\T}(0, q)$ to the $i$-th $\T$-fixed point 
\begin{align}\label{p_i} p_i := [\underbrace{0, ..., 1}_{i}, ..., 0] \in \PP ^{n-1} . \end{align}
 Define the $q$-series 
$\mu (q), R_0(q) \in \QQ[[q]]$ by the asymptotic expansion
\begin{align*}   I(0, q)|_{p_i}  \equiv e^{\mu (q)\lambda_i/z }  (R_0(q) + O(z )) , \end{align*}
where $\equiv$ means the equality after the specialization
 \begin{align}\label{roots} \lambda _i = \lambda _0 \exp {(2\pi i\sqrt{-1}/n)}, \ \ i=1,..., n.\end{align}
For the existence of the asymptotic expansion, see \eqref{asymJ}.
 Denote by  $$\underline{I}_{\T}$$ the specialization of $I_{\T}$
with \eqref{roots}.

For $k=0, 1, ..., n-1$, define the initial constants $C_k (q)\in \QQ[[q]]$ of form $1+O(q)$  inductively by the requirements
\[ C_k(q)H^k  = B_k +  O(1/z)  \] 
 in the following Birkhoff factorization procedure:
\begin{multline*} B_0 :=\underline{I}_{\T} (0, q),\   B_1:= (H+zq\frac{d}{dq}) \frac{B_0}{C_{0}(q)}, \cdots , \\
\ B_k   := (H + zq\frac{d}{dq}) \frac{B_{k-1}}{C_{k-1}(q)}, 
  \cdots, \ B_{n-1}:= (H + zq\frac{d}{dq})  \frac{B_{n-2}}{C_{n-2}(q)}. \end{multline*}
There is an interpretation of $C_k$ as a $\T$-equivariant quasimap invariant (see Remark \ref{C_b}).

\medskip

Now we are ready to state one of two main results of this paper.

\begin{Thm}\label{Main}
\begin{multline*}  \lan\; \ran ^{0+}_{1,0}   =  -\frac{3(n-1-r)^2+n-r-3}{48} \log (1 - q\prod _{a=1}^r  l_a ^{l_a}) \\
  - \frac{1}{2}\sum _{k=0}^{n-2-r} \binom{n-r-k}{2}\log C_{k} (q) .  \end{multline*}
\end{Thm}

\medskip

There is a following wall-crossing formula conjectured in \cite{CKg} and proven in \cite{CKw}.
Define  $I_0:=C_0$ and $ I_1$ by the $1/z$-expansion of
\[ I_{\T}|_{\lambda = t = 0} = I_0 + I_1/z + O(1/ z^2). \]

\begin{Thm}\cite{CKw} \begin{multline}\label{Wall} 
  \lan\; \ran ^{\infty}_{1,0}  |_{q^d \mapsto q^d \exp (\int _{d[\mathrm{line}]} \frac{I_1}{I_0} ) }  - \lan\; \ran ^{0+}_{1,0}    
   \\ = \frac{1}{24}\chi _{\mathrm{top}} (X) \log I_0  + \frac{1}{24} \int _X \frac{I_1}{I_0} c_{\dim X -1} (TX). \end{multline}
Here $\chi _{\mathrm{top}} (X)$ is the topological Euler characteristic of $X$ and $c_{\dim X -1} (TX)$
is the $(\dim X -1)$-th Chern class of the tangent bundle $TX$.
\end{Thm}

\medskip

Without any usage of the reduced Gromov-Witten invariants, Theorem \ref{Main} 
combined with the wall-crossing formula \eqref{Wall} reproves
the following mirror theorem of Zinger and Popa for  Calabi-Yau complete intersections in projective spaces. 

\begin{Thm}\label{Zinger}\cite{Zg1, Popa1}
\begin{multline*}  \lan\; \ran ^{\infty}_{1,0}  |_{q^d\mapsto  q^d \exp (\int _{d[\mathrm{line}]} I_1/I_0 )}  = 
\frac{1}{24} \chi _{\mathrm{top}}(X) \log I_0
+ \frac{1}{24} \int _X \frac{I_1}{I_0} c_{\dim X -1} (TX) 
 \\  - \frac{3(n-1-r)^2+n-r-3}{48} \log (1 - q\prod _{a=1}^r  l_a ^{l_a})  \\  - 
\frac{1}{2}\sum _{k=0}^{n-2-r} \binom{n-r-k}{2}\log C_{k} (q) . \end{multline*}
\end{Thm}

\bigskip

The above three theorems are logically independently proven and 
 any pair of them implies the remaining theorem.
 Theorem \ref{Main} combined with Theorem \ref{Wall} answers the question raised by Marian, Oprea, and Pandharipande in
\S 10.2 of \cite{MOP}.

 \medskip

The strategy to prove Theorem  \ref{Main} consists of two steps.
The first step is quite general and conceptual. We obtain Theorem \ref{G1}, one of two main results of this paper.
The theorem is a quasimap version of Givental's expression \cite{Elliptic}
of the elliptic Gromov-Witten generating function for a smooth projective toric variety twisted by a vector bundle. 
The latter expression  is given in terms of the equivariant Frobenius structure of the equivariant
quantum cohomology.
One may regard Theorem \ref{G1} as  a mirror theorem
for elliptic quasimap invariants for Calabi-Yau complete intersections in toric varieties (as well as in partial flag varieties, see 
Remark \ref{flag}), in the following sense.

Whenever one computes the RHS in Conjecture \ref{GConj} as a closed form, one obtains a closed form of the mirror theorem.
 Inspired by \cite{Zg1, Popa1}, we accomplish the computation for Calabi-Yau complete intersections in projective spaces.
This second step is completely algebraic. We will, however, see that  the geometric natures of various generating functions of quasimap invariants
make the step crucially simple.

\subsection{Acknowledgments} 
We thank Ionu\c t Ciocan-Fontanine and Jeongseok Oh for useful discussions. 
The research of B.K. and H.L. 
was partially supported by the NRF grant 2007-0093859. 

%%%%%
%%%%%
%%%%%   
%%%%%
%%%%%
%%%%%
%%%%%
%%%%%
%%%%%   Section II
%%%%%
%%%%%
%%%%%
%%%%%
%%%%%
%%%%%   
%%%%%
%%%%%
%%%%%

\section{Localized Elliptic Expression}\label{local_expression_section}

Let $\G$ be a complex reductive group and let $V$ be a finite dimensional
representation space of $\G$. Let $\theta$ be a character of $\G$ such that the semistable locus $V^{ss}(\theta)$ with
respect to $\theta$ has no non-trivial isotropy subgroup of $\G$.
Following the twisted theory as in \cite[\S 7]{CKg0}, we assume that a complex torus $\T$ acts on
a vector space $V$ and this action commutes with the $\G$ action on $V$. 
Assume furthermore, the induced action on $Y:= \VtG$ allows
only finitely many 0-dimensional and 1-dimensional $\T$-orbits. 

Let $E$ be a  $\G\times \T$-representation space.
Let $s$ be a $\G$-equivariant map from $V$ to $E$ 
whose zero locus  $W$ has only locally complete intersection singularities.
Assume that the semistable locus $W^{ss}(\theta)$ is nonsingular.

Denote $X= \WtG$. For $\beta \in \mathrm{Hom}_{\ZZ} (\mathrm{Pic}^\G V, \ZZ)$ 
let \[ Q_{g, k} (X, \beta) \text{ (resp. } Q_{g, k} (Y, \beta)) \]  be the moduli space of $k$-pointed genus $g$
stable quasimaps to $X$ (resp. $Y$) of class $\beta$. Denote by $f$ the universal map
from the universal curve $\mathcal{C}$ to the stack quotient $[V/\G]$:
\[ \xymatrix{   \mathcal{C} \ar[r]^{f} \ar[d]_{\pi} & [V/\G] \\
   Q_{g, k}(Y, \beta )  &    }  .\]
                     
                     Denote \[ \tilde{E} := [(E\times V) /\G] \]
                     which is a vector bundle on
$[V/\G]$. Note that  $s$ induces a section of a coherent sheaf $\pi _* f^* \tilde{E}$.
Assume that   for $g=0$ 
and also for $g=1, k=0, \beta \ne 0$, 
\[ R^1\pi _* f^* \tilde{E} = 0 . \]
For example this is the case when $E$ is a  sum $\oplus _a E_a$ of 1-dimensional $\G\times \T$-representations $E_a$
with $\G$ weight $m_a\theta$ for some positive integers $m_a$.

Let $\iota$ denote the closed immersion of $Q_{g, k} (X, \beta )$ into $Q_{g, k} (Y, \beta )$. 
By the functoriality in \cite{KKP}
we have 
\begin{equation}\label{functorial} \iota _* [Q_{g, k} (X, \beta )]^{\vir} = \re (\pi _* f^* \tilde{E}) \cap [Q_{g, k} (Y, \beta )]^{\vir} \end{equation}
for $g=0$, $k=2, 3,  ...$ and also for $g=1$, $k=0$, $\beta \ne 0$. In this paper
we study $[Q_{1, 0} (X, \beta )]^{\vir} $ using the (obvious $\T$-equivariant version of) RHS of \eqref{functorial}.

\subsection{Genus zero theory}

We introduce the definitions of various generating functions of
rational quasimap invariants with the ordinary markings.
We prove the relation \eqref{Dr} which will be needed later.

First we set the notation for the cohomology basis and its dual basis.
Let $\{ p_i \}_i$ be the set of $\T$-fixed points of $Y$ and
let $\phi _i$ be the \lq\lq delta" basis of $H^*_\T(Y)\ot \QQ (\lambda)$, that is,
\[ \phi_i |_{p_j} = \left\{ \begin{array}{rl} 1 & \text{if } i=j  \\
                                    0 & \text{if } i\ne j  \, . \end{array}\right. \] 
Let $\phi ^i$ be the dual basis with respect to the {\em $E$-twisted} $\T$-equivariant Poincar\'e pairing, i.e.,
\[ \int _Y \phi _i \phi ^j \ret (\tilde{E}|_Y) = \left\{ \begin{array}{rl} 1 & \text{if } i=j  \\
                                    0 & \text{if } i\ne j   \, ,\end{array}\right.  \]
where $\ret (\tilde{E}|_{Y})$ is the $\T$-equivariant Euler class of $\tilde{E}|_Y$. 

We assume that, for every $i$, $\ret (\tilde{E}|_{p_i})$ is {\em invertible} in $\QQ (\lambda)$ so that
the twisted Poincar\'e pairing is a perfect pairing on $H^*_\T(Y)\ot \QQ (\lambda)$.
Note that \[ \phi ^i = e_i\phi _i , \text{ where } 
e_i := \frac{1}{\int _Y \phi _i \phi _i \re ^\T (\tilde{E}|_{Y}) } = \frac{\ret(T_{p_i}Y)} {\ret (\tilde{E}|_{p_i})}.\]

Integrating along the twisted virtual fundamental class $$\ret (\pi_*f^* \tilde{E})\cap [Q_{0, k}(Y, \beta)]^{\vir}$$ we define correlators
$\lan ... \ran _{0, k, \beta}^{0+}$ as follows. For $\gamma _i \in H^*_\T(Y)\ot \QQ (\lambda )$,
\[ \lan \gamma _1\psi  ^{a_1} , ..., \gamma _k\psi  ^{a_k} \ran _{0, k, \beta}^{0+} := 
\int _{\ret (\pi_*f^* \tilde{E})\cap [Q_{0, k}(Y,\beta)]^{\vir}} \prod _i ev_i^*(\gamma _i)\psi _i ^{a_i} ,\]
where $\psi _i$ is the psi-class associated to the $i$-th marking and $ev_i$ is the 
$i$-th evaluation map.

Let $$Q_{0, k} (Y, \beta) ^{\T, p_i}$$  be the $\T$-fixed part of $Q_{0, k} (Y, \beta)$ 
whose elements have domain components only over $p_i$. 
Integrating along  the localized cycle class 
\[ \frac{\ret(\pi _*f^* \tilde{E}) \cap [Q_{0, k} (Y, \beta) ^{\T, p_i}]^{\vir}}{\ret (N^{\vir}_{Q_{0, k} (Y, \beta) ^{\T, p_i}/
Q_{0, k}(Y, \beta)} )} \]
we define $\lan ... \ran _{0, k, \beta}^{0+, p_i}$ and  $\lla ... \rra _{0, k}^{0+, p_i}$ as follows:
\begin{align*}  & \lan \gamma _1\psi  ^{a_1} , ..., \gamma _k\psi  ^{a_k} \ran _{0, k, \beta}^{0+, p_i}  :=
\int _{\frac{\ret(\pi _*f^* \tilde{E})\cap [Q_{0, k} (Y, \beta) ^{\T, p_i}]^{\vir}}{\ret (N^{\vir}_{Q_{0, k} (Y, \beta) ^{\T, p_i}/
Q_{0, k}(Y, \beta)} )}} \prod _i ev_i^*(\gamma _i)\psi _i ^{a_i} \ \ ; \\
&  \lla \gamma _1\psi  ^{a_1} , ..., \gamma _k\psi  ^{a_k} \rra _{0, k}^{0+, p_i} \\
& := \sum _{m, \beta} \frac{q^{\beta}}{m!}
 \lan    \gamma _1\psi  ^{a_1} , ..., \gamma _k\psi  ^{a_k} , t, ..., t  \ran_{0, k+m, \beta}^{0+, p_i} , \text{ for } t \in H_{\T}^* (Y)\ot \QQ (\lambda ) \ ,
      \end{align*} 
      where $q$ is a formal Novikov variable.

In what follows, let $z$ be a formal variable. 
We will need the following  $\T$-local generating functions:
\begin{align*} D_i & : = e_i  \lla 1, 1, 1 \rra ^{0+, p_i}_{0, 3 } = 1 + O(q) \ \ ; \\
  u_i & :=  e_i  \lla  1, 1 \rra_{0, 2}^{0+, p_i}    =  t|_{p_i} + O(q)  \ \ ; \\
 S_t^{0+, p_i} (\gamma ) & := e_i
 \lla \frac{1}{z-\psi } , \gamma \rra_{0, 2}^{0+, p_i} = e^t\gamma |_{p_i} + O(q) \\ 
        & \text{ for } \gamma \in H^*_{\T}(Y)\ot \QQ (\lambda ) [[q]]  \ \ ; \\
 J^{0+, p_i} &:= e_i  \lla \frac{1 }{z(z-\psi ) }\rra _{0,1}^{0+, p_i} =  e^t |_{p_i} + O(q) \ ,
\end{align*} 
where the unstable terms of $J^{0+, p_i}$ are defined by the quasimap graph spaces $QG^{0+}_{0, 0, \beta} (Y)$ as in \cite{CK, CKg0}
so that 
\[ J^{0+, p_i}|_{t=0}  = J^{0+} |_{t=0,\, p_i} \]
(see \S 5 of \cite{CKg0} for the definition of $J^{0+}$). 
Here the front terms $e_i$ are inserted as the  class $E$-Poincar\'e dual to $\phi_i |_{p_i} = 1$.
The parameter $z$ naturally appears as the $\CC ^*$-equivariant parameter in the graph 
construction (see \S 4 of \cite{CKg0}). It is originated from the $\CC ^*$-action on $\PP ^1$.

Denote by $QG^{0+}_{0, k, \beta } (Y) $ the quasimap graph spaces (see \cite{CKg0})
and by $$QG^{0+}_{0, k, \beta } (Y)^{\T, p_i} $$ the $\T$-fixed part of $QG_{0, k, \beta } (Y) $ 
whose elements have domain components only over $p_i$. 
Further, we define invariants and generating functions on the graph spaces:
for $\gamma _i \in H^*_{\T}(Y)\ot H^*_{\CC ^*} (\PP ^1) \ot \QQ (\lambda )$
\begin{align*}
& \lan \gamma _1\psi  ^{a_1} , ..., \gamma _k\psi  ^{a_k} \ran _{k, \beta}^{QG^{0+}, p_i} :=
\int _{\frac{\ret(\pi _*f^* \tilde{E})\cap [QG_{0, k, \beta } (Y) ^{\T, p_i}]^{\vir}}{\ret (N^{\vir}_{QG_{0, k, \beta} (Y) ^{\T, p_i}/
QG_{0, k, \beta }(Y)} )}} \prod _i ev_i^*(\gamma _i)\psi _i ^{a_i} \ \ ; \\
&  \lla \gamma _1\psi  ^{a_1} , ..., \gamma _k\psi  ^{a_k} \rra _{k}^{QG^{0+}, p_i} \\
& := \sum _{m, \beta} \frac{q^{\beta}}{m!}
 \lan    \gamma _1\psi  ^{a_1} , ..., \gamma _k\psi  ^{a_k} , t, ..., t  \ran_{k+m, \beta}^{QG^{0+}, p_i} , \text{ for } t \in H_{\T}^* (Y)\ot \QQ (\lambda ) .
\end{align*}
Here we denote by $ev_i$ the $i$-th evaluation map to $Y\times \PP ^1$ from the quasimap graph spaces
and regard $t$ also as the elements $t\ot 1$  in $H^*_{\T}(Y)\ot H^*_{\CC ^*} (\PP ^1) \ot \QQ (\lambda )$.

In what follows, let ${\bf p}_\infty$ be the equivariant cohomology class 
$H^*_{\CC ^*} (\PP ^1)$ defined by the requirements
\[ {\bf p}_{\infty} |_{0} = 0, \ {\bf p}_{\infty} |_{\infty} = - z .\]

\begin{Prop}\label{fact}
\begin{equation}\label{Der_Dr} J^{0+, p_i} = S_t ^{0+, p_i} (P^{0+, p_i}) , \end{equation}
where 
\[ P^{0+, p_i} :=  e_i  \lla 1 \ot {\bf p}_\infty \rra _{1 }^{QG^{0+}, p_i} \]
\end{Prop}

\begin{proof} The proof is completely parallel to the poof of Theorem 5.4.1 of \cite{CKg0}.
Fix the number of markings and the degree class $\beta$ and then apply the $\CC^*$-localization to the definition of $P ^{0+, p_i}$. 
\end{proof}

By the uniqueness lemma in \S 7.7 of \cite{CKg0}, 
\begin{equation}\label{localS}  S_t^{0+, p_i} (\gamma ) = e^{u_i/z} \gamma |_{p_i} . \end{equation}
Hence Proposition \ref{fact} gives
the expression \[ J^{0+, p_i} = e^{u_i/z} (r_{i, 0} + O(z)), \] where 
$r_{i, 0}\in \QQ (\lambda ) [[t, q]]$ is the constant term of $P^{0+, p_i}$ in $z$.

\begin{Cor} The equality 
\[ \log J^{0 +, p_i} = u_i/z + \log r_{i, 0} + O(z)  \ \ \in \QQ (\lambda ) ((z)) [[t, q]] \]
holds  as  Laurent series of $z$ over the
coefficient ring $\QQ (\lambda)$ in each power expansion of $t$ and $q$, after 
regarding $t$ as a formal element.
\end{Cor}
\begin{proof} It is clear that both side belong to $\QQ (\lambda)((z))[[t, q]]$.  \end{proof}

\begin{Cor}
\begin{align}\label{Dr} D_i |_{t=0}= \frac{1}{r_{i, 0}|_{t=0}}  . \end{align}
\end{Cor}

\begin{proof} By \eqref{Der_Dr} at $t=0$, \eqref{localS} with $\gamma =1$, and  the definition of $J^{0+, p_i}$ we see that
\begin{equation}\label{Iexp} J^{0 +, p_i}  = e^{u_i |_{t=0}/z} P^{0+, p_i} |_{t=0}  +  \frac{t}{z}S_{t=0}^{0+, p_i}(1) + O(t^2) . \end{equation}
Also by \eqref{localS} with $\gamma = 1$ and the definition of $S_t^{0+, p_i}$ we see that
\begin{equation}\label{Sexp} S_t^{0+, p_i} = e^{u_i|_{t=0}/z} (1 + \frac{t}{z} (D_i|_{t=0}) ) + O(t^2) . \end{equation}
The multiplication of $e^{-u_i|_{t=0}/z}$ and \eqref{Der_Dr} after the replacements of \eqref{Iexp}, \eqref{Sexp} gives
\[   P^{0+, p_i} |_{t=0}  + \frac{t}{z} + O(t^2) = P^{0+, p_i} |_{t=0} (1 + D_i|_{t=0} \frac{t}{z}  )   + O(t^2) . \] 
Now the comparison of the $\frac{t}{z}$-coefficient yields \eqref{Dr}.
 \end{proof}

\subsection{Insertions of $0+$ weighted markings}

To break the symmetry of the localization computation for the virtual fundamental classes of the elliptic quasimap moduli 
spaces we will need to introduce a marking. However, to keep the relation \eqref{functorial} even with markings for $g=1$, we will use 
the infinitesimally (i.e., $0+$) weighted markings. 

Denote by $$Q^{0+, 0+}_{g, k|m}(Y, \beta) \text{ (resp. } QG^{0+, 0+}_{0, k|m, \beta} (Y))  $$
the (resp. graph) moduli space of genus $g$ (resp. genus zero), degree class $\beta$ stable quasimaps to $Y$ with
ordinary $k$ pointed markings and infinitesimally weighted $m$ pointed markings (see \S 2, \S 5  of \cite{BigI}).   
They are isomorphic to the universal curve $\mathcal{C}$ of $Q^{0+}_{g, k|m-1}(Y, \beta)$ (resp $QG^{0+}_{0, k|m-1, \beta} (Y)$).
Denote by $$ Q^{0+, 0+}_{g, k|m}(Y, \beta )^{\T, p_i}, \text{  (resp. } QG^{0+, 0+}_{0, k|m, \beta} (Y)^{\T, p_i}) $$
the $\T$-fixed 
part of $Q^{0+, 0+}_{g, k|m}(Y, \beta)$, (resp. $QG^{0+, 0+}_{0, k|m, \beta} (Y)$) whose domain components are
only over $p_i$.

For $\gamma_i \in H^*_{\T} (Y)\ot \QQ (\lambda )$, $\tilde{t}, \delta _j \in H^*_{\T} ([V/\G ], \QQ )$ denote

\begin{align*}  & \lan \gamma _1\psi  ^{a_1} , ..., \gamma _k\psi  ^{a_k} ;  \delta _1, ..., \delta _m \ran _{0, k|m, \beta}^{0+, 0+}  \\ & : =
\int _{\ret(\pi _*f^* \tilde{E})\cap [Q^{0+, 0+}_{0, k|m} (Y, \beta) ]^{\vir}} \prod _i ev_i^*(\gamma _i)\psi _i ^{a_i} \prod _j \hat{ev}_j ^* (\delta _j) \ \ ; \\
&  \lla \gamma _1\psi  ^{a_1} , ..., \gamma _k\psi  ^{a_k}  ;  \delta _1, ..., \delta _m \rra _{0, k}^{0+, 0+} \\
& := \sum _{m', \beta} \frac{q^{\beta}}{m'!}
 \lan    \gamma _1\psi  ^{a_1} , ..., \gamma _k\psi  ^{a_k} ; \delta _1, ..., \delta _m, \tilde{t}, ..., \tilde{t}  \ran_{0, k|m+m', \beta}^{0+, 0+} \ \  ; \\ %%%%%%%%%%%%%%%%
 & \lan \gamma _1\psi  ^{a_1} , ..., \gamma _k\psi  ^{a_k} ;  \delta _1, ..., \delta _m \ran _{0, k|m, \beta}^{0+, 0+, p_i}  \\ & :=
\int _{\frac{\ret(\pi _*f^* \tilde{E})\cap [Q^{0+, 0+}_{0, k|m} (Y, \beta) ^{\T, p_i}]^{\vir}}{\ret (N^{\vir}_{Q^{0+, 0+}_{0, k|m} (Y, \beta) ^{\T, p_i}/
Q^{0+, 0+}_{0, k|m}(Y, \beta)} )}} \prod _i ev_i^*(\gamma _i)\psi _i ^{a_i} \prod _j \hat{ev}_j ^* (\delta _j) \ \ ; \\
&   \lla \gamma _1\psi  ^{a_1} , ..., \gamma _k\psi  ^{a_k} ; \delta _1, ..., \delta _m \rra _{0, k|m}^{0+, 0+, p_i} \\
& := \sum _{m', \beta} \frac{q^{\beta}}{m'!}
 \lan    \gamma _1\psi  ^{a_1} , ..., \gamma _k\psi  ^{a_k} ; \delta_1, ...., \delta _m, \tilde{t}, ..., \tilde{t}  \ran_{0, k|m+m', \beta}^{0+, 0+, p_i}  
 \ \   ,
 \end{align*}
where $\hat{ev}_j$ is the evaluation map to $[V/\G]$ at the $j$-th infinitesimally weighted  marking.
Here and below by {\em double brackets with superscript $0+, 0+$}, denote the sum over all degree class $\beta$ and all possible $\tilde{t}$ insertions
{\em only at the infinitesimally weighted markings}.
Similarly we define \[ \lan \cdots ; \cdots \ran^{QG^{0+, 0+}}_{k|m, \beta} \text{ and } \lla \cdots ; \cdots \rra ^{QG^{0+, 0+}, p_i}_{k|m} . \]

Consider
\begin{align*}
\mathds{S}(\gamma) & : =  \sum _{i}  \phi ^i  \lla \frac{\phi _i}{z-\psi} , \gamma \rra _{0, 2}^{0+, 0+}\ \ ;  \\
\mathds{V}_{ii} (x, y)  & := 
\lla \frac{\phi _i}{x- \psi } ,  \frac{\phi _i}{y - \psi } \rra _{0, 2}^{0+, 0+} 
= \frac{1}{e_i(x+y)} + O(q) \ \ ; \\
 \mathds{U}_i & :=  e_i 
 \lla 1, 1 \rra_{0, 2}^{0+, 0+,  p_i}    =  \tilde{t}|_{p_i} + O(q)  \ \ ; \\
\mathds{S}_i^{0+, p_i} (\gamma) & := e_i 
 \lla \frac{1}{z-\psi} , \gamma \rra _{0, 2}^{0+, 0+,  p_i}   = e^{\tilde{t}/z}\gamma |_{p_i}+ O(q)\ \ ;\\
\mathds{J}^{0+, p_i} &:= e_i 
 \lla \frac{1}{z(z-\psi)}  \rra_{0, 1 }^{0+, 0+, p_i} = e^{\tilde{t}}|_{p_i} + O(q) = J^{0+, p_i}|_{t=0} + O(\tilde{t}) .
 \end{align*}
(Here $e_i^2 \mathds{V}_{ii}$ at $\tilde{t}=0$ coincides with $V^{0+}_{t=0}|_{p_i}$ of \cite{CKg}.)

As before, 
\begin{align} \mathds{S}_i^{0+, p_i} (\gamma)  &= e^{\mathds{U}_i/z} \gamma |_{p_i} \ \ ; \nonumber \\
\label{asymJ}  \mathds{J}^{0+, p_i} & = e^{\mathds{U}_i/z} (\sum _{k=0}^m \mathds{R}_{i,k}z^k + O(z^{m+1})) 
\end{align} 
for some unique $\mathds{R}_{i,k} \in \QQ (\lambda)[[\tilde{t}, q]]$ (after regarding $\tilde{t}$ as a formal 
element).

\subsection{Birkhoff factorization}\label{inf_I}

In this subsection we do not need to assume that the $\T$ action on $Y$ has isolated fixed points. 
Therefore in this subsection $\{ \phi _i\} _i $ will denote  any chosen basis of $H^*_{\T}(Y) \ot \QQ (\lambda )$ with
its $E$-Poincar\'e dual basis $\{\phi ^i\}_{i}$.

Denote by $\mathds{I}$ the infinitesimal $I$-function $\mathds{J}^{0+, 0+}$ defined in \cite{BigI}.
The $\mathds{S}$ introduced in the previous section is, by the very definition, the infinitesimal $S$-operator $\mathds{S}^{0+, 0+}$ defined in \cite{BigI}. Hence
\[ \mathds{I} := \mathds{J}^{0+, 0+} \ \text{ and }  \ \mathds{S} := \mathds{S}^{0+, 0+} . \]
For $\gamma \in H^*_{\T}(Y)$, $\tilde{\gamma}\in H^*_{\T} ([V/\G])$ denotes a lift of $\gamma$, i.e., $\tilde{\gamma} |_Y = \gamma$.

 Let ${\bf p}_0$ be the equivariant cohomology class 
$H^*_{\CC ^*} (\PP ^1)$ defined by 
\[ {\bf p}_{0} |_{0} = z, \ {\bf p}_{0} |_{\infty} = 0 .\]
Consider
$$P_{\tilde{\gamma}} :=  \sum _{i} \phi ^i  \lla \phi_i \ot \mathbf{p}_\infty ; \tilde{\gamma} \ot \mathbf{p}_0 \rra ^{QG^{0+, 0+}}_{0, 1|1} 
\in H^*_{\T}(Y)\ot \QQ (\lambda )[z] [[\tilde{t}, q]] $$ and its virtual 
$\CC^*$ localization factorization.  As in Proposition 4.3 of \cite{BigI}, 
there is a Birkhoff factorization  
\begin{align}\label{Bfactor} z \partial _{\gamma} \mathds{I} :=  z\frac{d}{ds}|_{s=0} \mathds{I}(\tilde{t}+ s\tilde{\gamma} )  = \mathds{S} (P_{\tilde{\gamma}}) . \end{align}

Since $P_{\tilde{\gamma}} = \gamma + O(q )$, the factorization \eqref{Bfactor}  implies that for each $\tilde{t}$
there is a unique expression of $\mathds{S} (\gamma)$ as a linear combination 
of $\partial _{\phi _i} \mathds{I}$ with coefficients in $\QQ (\lambda)[z][[q]]$.
Hence we conclude the following Proposition.

\begin{Prop}\label{BP}  For each $\tilde{t}$, there are unique coefficients $ a_i (z, q) \in \QQ (\lambda)[z][[q]]$ making
 \[ \sum _{i} a_i (z, q)  z\partial _{\phi_i} \mathds{I}  = \gamma + O(1/z) .\] Furthermore
 LHS coincides with  $\mathds{S}(\gamma)$.
\end{Prop}

\subsection{Genus one theory}\label{one_theory}

From now on, we assume that the Calabi-Yau condition holds, i.e., 
\[ c_1(Y) - c_1 (\tilde{E}|_Y ) = 0 \text{ in } H^2(Y, \QQ) . \] 
We apply Givental's localization method \cite{Elliptic} to express 
a genus one generating function in terms of the genus zero generating functions.

Consider the genus one generating function with one insertion at an infinitesimally (i.e., $0+$) weighted marking:
\begin{align*}  \lan  ;  \tilde{\gamma} \ran ^{0+, 0+}_{1, 0|1} 
&:= \sum _{d=1}^{\infty} q^d \lan ; \tilde{\gamma} \ran ^{0+, 0+}_{1, 0|1, d} ,  \end{align*}
where $\tilde{\gamma} \in H^2_{\T} ([V/\G], \QQ)$. We will study the generating function using the virtual 
$\T$ localization. 

In the following conjecture, 
 $c_i(\lambda)$ denotes the element in  $\QQ (\lambda)$ uniquely determined by 
\begin{align*} 1+ c_i(\lambda) \re (\mathbb{E})  
&  = \frac{\re ^\T(\mathbb{E}^\vee \ot T_{p_i}Y ) \re ^\T ( \tilde{E} |_{p_i})}
{\re ^\T(T_{p_i} Y) \re ^\T (\mathbb{E}^\vee \ot \tilde{E}|_{p_i} )} ,
\end{align*} where
$\mathbb{E}$ is the Hodge bundle on the moduli stack $\overline{M}_{1,1}$ of stable one pointed genus 1
curves.

\begin{Conj}\label{GConj}  For $\tilde{\gamma}  \in H^2_{\T} ([V/\G], \QQ )$
\begin{align}\label{Gexp}  &  \lan  ; \tilde{\gamma} \ran ^{0+, 0+}_{1, 0|1}    =  \sum _i
 q_{\tilde{\gamma} }\frac{\partial}{\partial q_{\tilde{\gamma} }}  \left(-\frac{\log \mathds{R}_{i,0}|_{\tilde{t}=0}}{24}   +  
c_i(\lambda) \frac{\mathds{U}_i|_{\tilde{t}=0}}{24} \right) \\ 
&  \label{loop} 
+ \frac{1}{2} \sum _i  
\left( \partial _{\tilde{\gamma} } \mathds{U}_i |_{\tilde{t}=0} \right)
\lim_{(x,y)\ra (0,0)}\left((e^{-\mathds{U}_i(\frac{1}{x} + \frac{1}{y})} e_i\mathds{V}_{ii}(x, y) - \frac{1}{x+y} )  |_{\tilde{t}=0} \right) ,
\end{align}
where $q_{\tilde{\gamma} }\frac{\partial}{\partial q_{\tilde{\gamma} }}$ acts on  $q^{\beta}$ by
$q_{\tilde{\gamma} } \frac{\partial}{\partial q_{\tilde{\gamma} }} q^{\beta} =  q^{\beta} \int _{\beta} \tilde{\gamma}  $.
\end{Conj}

We prove Conjecture \ref{GConj} in the following toric setting.
Let $Y$ be a projective smooth toric variety defined by
a fan $\Sigma$. Let $\Sigma (1)$ be the collection of all 1-dimensional cones $\rho$ in $\Sigma$ and let 
$V=\CC ^{\Sigma (1)}$. Then $Y$ is also given by a GIT quotient $\CC ^{\Sigma (1)}/\!\!/_{\theta}\G$
for the complex torus $\G = (\CC ^*)^{|\Sigma (1) |- \dim Y}$ and some character $\theta$ of $\G$. Denote by $\T$ the big torus $(\CC ^*)^{\Sigma (1)}$.
Let $E$ and $W$ be as in the beginning of \S\ref{local_expression_section}.

\begin{Thm}\label{G1}
Conjecture \ref{GConj} holds true for the toric setting.
\end{Thm}

\subsection{The proof of Theorem \ref{G1}}\label{toric_case}
There is a natural 1-1 correspondence between the $\T$ fixed points of $Y$ and the maximal cones of $\Sigma$.
For a maximal cone $\sigma$, denote by $p_\sigma$ the corresponding $\T$ fixed point.
The $\T$-fixed loci of $Q^{0+, 0+}_{1, 0| 1}(Y, \beta)$ are divided into two types according to whether
the domain curves are irreducible or not. 
A quasimap in $Q^{0+, 0+}_{1, 0| 1}(Y, \beta)^\T$ will be called  a {\em vertex type} over $p_\sigma$  
if all domain components of the quasimap are all over $p_\sigma$.
Otherwise, the quasimap will be called a {\em loop type}. The loop type quasimap is called a loop type over $p_\sigma$
if the marking of the quasimap is over $p_\sigma$. 

Let $Q_{vert, \sigma} ^\T$ be the substack of $Q^{0+, 0+}_{1, 0|1}(Y, \beta) ^\T$ consisting of the vertex types over $p_\sigma$.
Let $Q_{loop, \sigma} ^\T$ be the substack of $Q^{0+, 0+}_{1, 0|1}(Y, \beta ) ^\T$ consisting of the loop types  over $p_\sigma$.

By the virtual localization theorem, 
$ \lan  ; \tilde{\gamma} \ran ^{0+, 0+}_{1, 0|1}  $ is the sum of the localization contribution
$\mathbf{Vert}^{\tilde{\gamma}}_{\sigma}$ from all the vertex types over $p_\sigma \in Y^\T$ and the localization 
contribution $\mathbf{Loop}^{\tilde{\gamma}}_\sigma$
from all the loop types over $p_\sigma\in Y^\T$.
 That is,
\begin{align*}   \lan  ; \gamma \ran ^{0+, 0+}_{1, 0|1}    & := \sum _{\sigma} \mathbf{Vert}^{\tilde{\gamma}}_{\sigma}  
+ \sum _{\sigma} \mathbf{Loop}^{\tilde{\gamma}}_{\sigma}, \end{align*} 
where
\begin{align*}
\mathbf{Vert}^{\tilde{\gamma}}_{\sigma} & := \sum _{\beta \ne 0} q^{\beta}  \int _{[Q_{vert, \sigma} ^{\T}]^{\vir}} 
\frac{\ret(\pi_*f^* \tilde{E})|_{Q_{vert,\sigma}^{\T}} \hat{ev}^*_1 (\tilde{\gamma}) }{\ret(N^{\mathrm{vir}}_{Q_{vert, \sigma}^{\T}/Q^{0+, 0+}_{1,0|1}(Y, \beta)} ) } ,
\\ \mathbf{Loop}^{\tilde{\gamma}}_{\sigma}  & := \sum_{\beta \ne 0}  q^{\beta}  \int _{[Q_{loop, \sigma} ^{\T}]^{\vir}} 
\frac{\ret(\pi_*f^* \tilde{E} )|_{Q_{loop,\sigma}^{\T}}  \hat{ev}^*_1 (\tilde{\gamma}) }{e(N^{\mathrm{vir}}_{Q_{loop, \sigma}^{\T}/Q^{0+, 0+}_{1,0|1}(Y, \beta)} ) } . \end{align*}

The above loop term can be identified with \eqref{loop} 
by an argument completely parallel to the corresponding procedure in the proof of Theorem 2.1 of  \cite{Elliptic}.

The analysis of vertex terms needs a nontrivial modification to the corresponding procedure of \cite{Elliptic}
due to the appearance of diagonal classes $\Delta _J$ of $\overline{M}_{1, 0|d}$, where $J\subset [d]:=\{ 1, 2, ..., d\}$.
 Here $\Delta_J$ is the codimension $|J|-1$ cycle class represented by the locus where 
$0+$ weighted markings of $J$ coincide to each other.

Let $\mathbf{Vert}_{\sigma}$ be the $p_\sigma$-vertex part of $\lan\, \ran ^{0+}_{1, 0}$. 
Then by the divisor axiom for the infinitesimally weighted marking,  
$\mathbf{Vert}^{\tilde{\gamma}}_{\sigma} =   q_{\tilde{\gamma} } \frac{\partial}{\partial q_{\tilde{\gamma} }}   \mathbf{Vert}_{\sigma}  $.
Therefore it is enough to show 
\begin{align}\label{purevertex}    \mathbf{Vert}_{\sigma} =   -\frac{\log \mathds{R}_{\sigma,0}|_{\tilde{t}=0}}{24}   +  
c_{\sigma}(\lambda) \frac{\mathds{U}_{\sigma}|_{\tilde{t}=0}}{24}.  \end{align}

For $\rho\in \Sigma (1)$, let $\xi _{\rho}$ be the character of the $\G$ action on
the corresponding coordinate of $\CC ^{\Sigma (1)}$. Recall that $\xi _{\rho '} , \rho ' \not\subset\sigma$
form a basis of the character group of $\G$. Hence we may
let $\xi_\rho = \sum_{\rho' \not\subset\sigma}  a_{\rho , \rho '} \xi _{\rho '}$ for some unique integers $a_{\rho , \rho '}$.
For a curve class $\beta \in \Hom _{\ZZ}(\Pic ^{\G} V, \ZZ)$, denote by
$\beta (\rho)$ the integer value of
$\beta$ at the line bundle associated to $\xi _\rho$.

Let $\beta_{\sigma}$ be the set of all pairs $(\rho , j)$, $\rho \not\subset\sigma$, $j\in [\beta (\rho)]$.
Then the $\T$-fixed $p_{\sigma}$-vertex part of $Q_{1, 0}(Y, \beta)$ is the quotient of
$\overline{M}_{1, 0| \beta _{\sigma} }$  by a finite group of order $\prod _{\rho \not \subset \sigma} \beta (\rho )!$,

For $C\in \overline{M}_{1, 0| \beta _{\sigma} }$, we denote the marked point by $x_{(\rho ' , j)}$ attached to the index $(\rho' , j)\in \beta_{\sigma}$.  
Let $\hat{\bx}_{\rho '}$ for $\rho ' \not\subset\sigma$ denote the effective divisor $\sum_{j\in [\beta (\rho ')] }  x_{(\rho ' , j)}$
and let $\hat{\bx}_{\rho}$ for $\rho \in \Sigma (1)$ denote the divisor $\sum_{\rho '}  a_{\rho ,\rho '} \hat{\bx}_{\rho '}$
of $C$. Here for $\rho \not\subset \sigma$ with $\beta(\rho ')=0$ we set 
$\hat{\bx}_{\rho '} = 0$. 
Then the corresponding quasimap in $Q_{1, 0}(Y, \beta)$ is a pair $$(C, \{ \mathcal{O}_C( \hat{\bx}_{\rho} ), u_\rho \} _{\rho \in \Sigma (1)}  ), $$
where $u_\rho$ is the canonical section of $\mathcal{O}_C( \hat{\bx}_{\rho} )$ if $\rho \not\subset\sigma$, otherwise $u_\rho$ is zero.

Let $r$ be the dimension of $E$.
Decompose $E=\oplus _{i=1}^r E_i$ by  $1$-dimensional $\T\ti \G$-representations $E_i$. 
Denote by $\xi _i$  the  character of $\G$ associated to the $\G$ action on $E_i$. Then $\xi _i = \sum_{\rho ' \not\subset\sigma} b_{i, \rho '} \xi _{\rho '}$ for some
unique integers $b_{i, \rho '}$. Let
$\hat{\bx}_i = \sum _{\rho '\not\subset\sigma} b_{i, \rho '} \hat{\bx}_{\rho '}$. 

In below for a divisor $D = \sum_{i} a_i p_i $ of $C$ with $p_i\in C$, $a_i \in \ZZ$, we define $D^+ := \sum_{a_i > 0}  a_i p_i$ 
and $D^- :=\sum _{a_i<0} a_i p_i$.  By  the localization formula (see \S 5.4 of \cite{CKg}), note that 
\begin{multline*}
         \mathbf{Vert}_\sigma  = \sum _{d \ne 0} 
         \frac{q^{\beta}}{\prod_{\rho\not\subset\sigma}\beta(\rho)!}  \int _{\overline{M}_{1, 0|\beta_{\sigma} }} 
         (1+ c_{\sigma}(\lambda)    \re (\mathbb{E}) ) 
         F^{(1,0)}_{\sigma,\beta} , 
 \text{ where }   \\  F^{(1,0)}_{\sigma,\beta} = 
\prod _{\rho \subset \sigma} \frac{\re^{\T} (H^0(C, \mathcal{O}_{\hat{\bx}_{\rho}^-} (\hat{\bx}^+_{\rho}) \ot \CC _{\sigma, \rho})}
{\re^{\T} (H^0(C, \mathcal{O}_{\hat{\bx}_{\rho}^+} (\hat{\bx}^+_{\rho}) \ot \CC_{\sigma, \rho })} 
\prod _{i=1}^r \frac{\re ^{\T} (H^0(C, \mathcal{O}_{\hat{\bx}_{i}^+} (\hat{\bx}^+_{i}) \ot \tilde{E}|_{p_\sigma})}
{\re^{\T}(H^0(C, \mathcal{O}_{\hat{\bx}_{i}^-} (\hat{\bx}^+_{i}) \ot \tilde{E}|_{p_\sigma})}  ,
 \end{multline*}
where $\CC_{\sigma, \rho}$ denotes the 1-dimensional $\T$ subspace of $T_{p_{\sigma}}Y$ corresponding to
the facet of $\sigma$ complement to $\rho$.

 For nonnegative integers $g, m$, the above expression for $F^{(1,0)}_{\sigma, \beta}$ also defines $F^{(g,m)}_{\sigma, \beta}$ as 
 an element in $H^*(\overline{M}_{g, m | \beta _{\sigma}}, \QQ (\lambda ))$, which can be written 
 by a polynomial of diagonal classes and the psi classes $\hat{\psi}_k$ (associated to the $0+$ weighted $k$-th marking).

Let $d$ be a positive integer. For $J\subset [d]$, let $\hat{\psi }_{J}$ denote $\hat{\psi }_j |_{\Delta _J}$, for any $j \in J$.
Note that for $J_1\cap J_2 \ne \emptyset$, \[ \Delta _{J_1} \Delta _{J_2} = (- \hat{\psi}_{J_1\cup J_2} ) ^{|J_1\cap J_2 | -1} \Delta _{J_1\cup J_2} \]
in $H^*(\overline{M}_{g, m| d}, \QQ)$ (see \S 4.4 of \cite{MOP}).
For $j\in [d]$, define $\Delta _{j}$ to be the fundamental class.
For a partition $J=\{ J_1, ..., J_k\}$ of $[d]$ (i.e., $\emptyset\ne J_i \subset [d]$ and $\coprod_{i=1}^k J_i = [d]$), define \[
\Delta _{J} : = \Delta _{J_1} ... \Delta _{J_k} . \]
Then $F^{(g, m)}_{\sigma, \beta}$ can be written
$$F^{(g, m)}_{\sigma, \beta} = \sum _{\text{partion } J=\{ J_1, ..., J_k\} \text{ of } \beta_{\sigma} } a_J \Delta_J $$
as a linear sum of $\Delta_J$ over the coefficient ring
$\QQ(\lambda)[\hat{\psi} _{\bullet} | \bullet\in \beta_{\sigma} ]$.

Since $\re (\mathbb{E}) \hat{\psi} _{\bullet} = 0$ in $H^*(\overline{M}_{1, 0| \beta_{\sigma}})$ and
$\hat{\psi}_{\bullet} = 0$ in $H^*(\overline{M}_{0, 2| \beta_{\sigma} })$, we see that
\begin{align*} \int _{\overline{M}_{1, 0| \beta_{\sigma}}} \re (\mathbb{E})  F^{(1,0)}_{\sigma, \beta} 
& = \frac{ \text{ Coeff. of the const. term }  \text{ in } a_{\{ \beta _{\sigma} \}} }{24}  \\
       & = \frac{1}{24} \int _{\overline{M}_{0, 2| \beta_{\sigma}}} F^{(0,2)}_{\sigma, \beta }  . \end{align*}
This explains the last term of \eqref{purevertex}.

The verification of the first term in RHS of \eqref{purevertex} requires a further analysis of $F^{(g, m)}_{\sigma,\beta}$.
First, observe that the $j$-th cotangent line on $\overline{M}_{g, m| d }$ for $j\in [d]$ is  naturally isomorphic to
the $j$-th cotangent line on  $\overline{M}_{g, m| d-1}$ under the pullback of the forgetting map of
the last $0+$ weighted point. Therefore $\hat{\psi }_{\bullet} ^2 = 0$ $H^*(\overline{M}_{g, m| \beta_{\sigma}})$
and for any partition $J=\{ J_1, ..., J_k\}$ of $\beta_\sigma$
\begin{align}\label{int1}
\int _{\overline{M}_{1, 0| \beta _{\sigma} }} \Delta _J \hat{\psi} _{J_1}^{a_1} ... \hat{\psi} _{J_k} ^{a_k}  
= \left\{\begin{array}{cc} 1/24 & \text{ if } k=1, a_1 = 1 \\
                   0 & \text{ otherwise } \end{array} \right. ; \\ \label{int2}
        \int _{\overline{M}_{0, 3|\beta_{\sigma}}} \Delta _J \hat{\psi} _{J_1}^{a_1} ... \hat{\psi} _{J_k} ^{a_k}  
= \left\{\begin{array}{cc} 1 & \text{ if } a_1 = ... = a_k =1 \\
                   0 & \text{ otherwise } \end{array} \right. .
           \end{align}
Let  \[ A^{\beta} _{J_1, ..., J_k} := \text{ Coeff. of } \prod _{i=1}^k \hat{\psi} _{J_i}  \text{ in } a_{\{J_1, ..., J_k \}} . \]
When $k=1$, we denote  $A^{\beta}_{J_1, ..., J_k}$ simply by $A^{\beta}$.
Denote by $\beta_{J_i}$ the set of all pairs $(\rho , j)$ such that $j\in [ |J_i(\rho) | ] $, where $J_i (\rho) := \{ (\rho, j) \in J_i \}$.
Then note that 
\begin{equation}\label{prod}  A^{\beta}_{J_1, ..., J_k} = \prod _{i=1}^k A^{\beta _{J_i}} , \end{equation}
which follows from two properties:

1) $F_{\sigma , \beta}$ is a product of the $\T$-equivariant Euler classes 
of vector bundles with fibers
$H^0(C, \mathcal{O}_{D} (B))$ where $D$ is an effective divisor and
$B$ is a divisor of $C$.  Here supports of $D$ and $B$ are contained in $\beta_{\sigma}$.

2) Let 
$D=D_1+D_2$, where $D_1$, $D_2$ are effective and let  $B=B_1+B_2$.
Then in the K-group element 
\[ \mathcal{O}_{D_1+D_2} (B_1+B_2 ) = \mathcal{O}_{D_1} (B_1) \ot \mathcal{O}_{D_1} (B_2)+ 
\mathcal{O}_{D_2}(B_2) \ot \mathcal{O}_{D_2} (-D_1+B_1) . \]
Suppose that $\beta_{\sigma}$ is a disjoint union of $S_1, S_2$ 
such that supports of $D_i$, $B_i$ are in $S_i$ for each $i=1, 2$.
Then
\begin{multline*}  \re^{\T} (H^0 (C, \mathcal{O}_{D_1+D_2} (B_1+B_2 ) ))|_{S_1, S_2}  \\
= \re ^{\T} (H^0  (C, \mathcal{O}_{D_1} (B_1)  )) \re^{\T} (H^0 (C,  \mathcal{O}_{D_2}(B_2) )) , \end{multline*}
where the restriction to $S_1, S_2$ is defined to be letting $\Delta _J =0$ whenever there is 
$J_i$ in the partition $J$ such that $J_i$ intersects with $S_1$ and $S_2$ simultaneously.

By \eqref{int1}, \eqref{int2}, \eqref{prod} we note that
 \begin{align*}  & \int _{\overline{M}_{0, 3| \beta_{\sigma} }} F^{(0,3)}_{\sigma,\beta }  
  =   \sum _{k=1}^{\infty}
 \sum_{ \substack{\text{partition } J_1, ..., J_k \\ \text{ of } \beta _{\sigma} } } 
  \prod _{i=1}^k  A^{\beta_{J_i}}    \\
 & =   \sum _{k=1}^{\infty} \frac{1}{k!}
 \sum_{ \substack{ \text{ ordered partition }  \\ (J_1, ..., J_k)  \text{ of } \beta _{\sigma} } } 
  \prod _{i=1}^k  A^{\beta_{J_i}}    \\ 
& =  \sum _{k=1}^{\infty} \frac{1}{k!}
  \sum_{ \substack{\text{ ordered } \\ (\beta _{J_1}, ..., \beta_{J_k}) }} 
 \prod _{\rho\not\subset\sigma}  {\beta (\rho )\choose |J_1  (\rho )|, ...,  |J_k  (\rho )| } \prod _{i=1}^k  A^{\beta_{J_i}}   \\
   &=  \prod _{\rho \not\subset\sigma} \beta(\rho )!   \sum _{k=1}^{\infty} 
  \frac{1}{k!}  \sum_{ \substack{\text{ ordered } \\ (\beta _{J_1}, ..., \beta_{J_k}) } } 
  \prod_{i=1}^k   \left(\frac{24}{\prod _{\rho \not\subset \sigma} |J_i (\rho  )| ! } 
   \int _{\overline{M}_{1, 0 | \beta _{J_i} }}  F^{(1,0)}_{\sigma, \beta _{J_i}}  \right).
\end{align*}

Hence 
\[ \sum_{\beta \ne 0}  \frac{q^{\beta}}{\prod_{\rho\not\subset\sigma} \beta(\rho) !} \int _{\overline{M}_{1, 0| \beta_{\sigma} }} F^{(1,0)}_{\sigma, \beta }  
= \frac{\log D_{\sigma} |_{t=0}}{24}. \]
This combined with \eqref{Dr} verifies  the first term in RHS of \eqref{purevertex}.

\begin{Rmk}\label{flag}
By \S 5.9.2 of \cite{CKg},
it is clear that the above proof works also for Calabi-Yau zero loci of homogeneous vector bundles on partial flag varieties $Y$,
local toric varieties, local Grassmannians, and the total spaces of the cotangent bundles of partial flag varieties. 
\end{Rmk}

%%%%%%%%%%%%%%%%%%%
%%%%%%%%%
%%%%%%%%%
%%%%%%%%%
%%%%%%%%%
%%%%%%%%%
%%%%%%%%%
%%%%%%%%%      Section III
%%%%%%%%%
%%%%%%%%%
%%%%%%%%%

\section{Explicit Computations}\label{explicit_comp}

In this section we prove Theorem \ref{Main}.
From now on unless stated otherwise, let $\G=\CC ^*$, $\T=(\CC ^*)^n$, and 
let $\CC_{l_a}$ be the $1$-dimensional representation space of $\G$ with positive weight $l_a$.
Let $E=\bigoplus _{a=1}^r \CC _{l_a}$, with $\sum_{a=1}^r l_a = n$. We take the standard $\T$-action on $V$ and the $\T$-trivial action on $E$. 
This gives rise to a $\T$-equivariant vector bundle $\tilde{E}$ on $[V/\G]$. Choose a character $\theta$ such that
$Y:=V/\!\!/_{\theta}\G$ becomes $\PP ^{n-1}$. 
Under the natural isomorphism $\mathrm{Hom}_{\ZZ} (\mathrm{Pic}^\G V, \ZZ) \cong \ZZ$, we use 
a nonnegative integer $d$ instead of $\beta$. Let $p_i$ be the $i$-th $\T$ fixed point of $Y$ as in \eqref{p_i}.

\subsection{Birkhoff factorization revisited}

By \cite{Gequiv} (see also (5.3.1) of \cite{BigI}), 
\[ \mathds{I}_{\tilde{t}=0} =  I_{\T} |_{t=0}, \]
where LHS  and RHS are defined in \S \ref{inf_I} and \eqref{equi_I}, respectively.

We define the degrees of $\lambda, H, q$ as $$\deg \lambda _j = 1 = \deg H, \  \deg q = 0.$$
Then it is easy to check that, for $k=0, 1, ..., n-1$, 
the $1/z^k$-coefficient  $I_k$ of $$ I_{\T} |_{t=0}$$  is a homogeneous degree $k$ element in 
$\QQ[\sigma _1,..., \sigma _{n-1} , H] [[q]] $ satisfying 
\begin{equation}\label{I_k} I_k \in   \QQ[[q]] H^k   \text{ modulo } \eqref{roots}. \end{equation}
On the other hand \begin{equation*}\label{SH} \mathds{S}_{\tilde{t}=0}(H^k)= H^k + O(1/z) \text{ for } 0\le k \le n-1. \end{equation*}
{\em Throughout \S\ref{explicit_comp} we impose the condition \eqref{roots}.} 
After \cite{Zg1} we define an operation as follows.
\[ \text{ For } F \in \left(\frac{\QQ[H]}{(H^n - \lambda _0 ^n )} \right) [[1/z]] [[q]]  \text{ with } ((zq\frac{d}{dq} + H) F )|_{H=1, z=\infty, q=0} \ne 0 \]
let
\[ \mathfrak{B}(F) := \frac{(z q\frac{d}{dq} + H) F (z, H, q)}{ ((zq\frac{d}{dq} + H) F (z, H, q) )|_{H=1, z=\infty} }.\]
Consider $\mathfrak{B}^k\left(\frac{\underline{I}_\T|_{t=0}}{I_0}\right) $
and note that it is of form $H^k + O (1/z)$ and homogenous of degree $k$ if we put $\deg z = 1$.

\begin{Cor}
\begin{align}\label{ExplicitS} \mathds{S}_{\tilde{t}=0}(H^k) \equiv \mathfrak{B}^k\left(\frac{\underline{I}_\T|_{t=0}}{I_0}\right) ,  \ k=0, 1, ..., n-1 .\end{align}
Here we recall that $\equiv$ denotes the equality modulo relations \eqref{roots}.
\end{Cor}

\begin{proof}
Let $\tilde{H} \in H^*_{\T} ([\CC ^n / \CC ^*]) $ be the natural lift of $H$ and 
let $$\tilde{t} = \sum _{i=0}^{n-1} t_i \tilde{H}^i$$ with formal variables $t_i$. Then, there is the
($\T$-equivariant version of) derivative form formula of the big I-function as in  \S 5.3 of \cite{BigI}:
\begin{equation*}\mathds{I} (\tilde{t} ) =  \left(\exp (\sum _{i=0}^{n-1} \frac{t_i}{z} (z q \frac{d}{dq} + H)^i ) \right) I_{\T} |_{t=0},
\end{equation*}
 which shows that
 \begin{equation}\label{One_I_Many_Div} (z\partial _{H^i} \mathds{I} )  |_{\tilde{t}=0}  = (zq\frac{d}{dq} + H)^i (\mathds{I}|_{\tilde{t} =0}). \end{equation}
By \eqref{One_I_Many_Div}  and  Proposition \ref{BP}, in order to verify \eqref{ExplicitS}, it is enough to recall that both sides of \eqref{ExplicitS}
are of form $H^k + O(1/z)$.
 \end{proof}

Now consider an equivariant cohomology basis $$\{ 1, H:=c_1^{\T}(\mathcal{O}(1)), ..., H^{n-1}\}$$ of the $\T$-equivariant cohomology ring
\begin{align*} H^*_\T(\PP ^{n-1}) \cong \QQ [\lambda _1, ..., \lambda _n, h]/ (\prod _{i=1}^n  (h - \lambda _i) ) , \ \ \ 
                           H \mapsto  h . \end{align*}
Its $E$-twisted Poincar\'e metric modulo relations \eqref{roots} becomes, 
 \begin{align*}  g_{ij} := \left\{ \begin{array}{rl} \prod _a l_a & \text{if } i+j = n-1-r \\
                                                            \lambda_0^n \prod _a l_a & \text{if } i+j = 2n-1-r  \end{array} \right.
  \end{align*} for $0\le i, j \le n-1$.  Here we use the relation $H^n = - \prod_{j=1}^n (-\lambda _j) = \lambda _0^n$.

There is an  expression of $V$-correlators in terms of $S$-correlators by Theorem 3.2.1 of \cite{CKg}:
\begin{align*}  e_i \mathds{V}_{ii} (x, y)|_{\tilde{t}=0}  = \frac{1}{e_i}
\frac{\sum _j \mathds{S}_{z=x, \tilde{t}=0} (\phi_j)|_{p_i} \mathds{S}_{z=y, \tilde{t}=0} ( \phi ^j )|_{p_i}}{x+ y}  . \end{align*}
Hence 
\begin{multline}\label{EquivDelta}  e_i\mathds{V}_{ii} |_{\tilde{t}=0} \equiv  \frac{1}{(\prod l_a)e_i(x+y)}   \left(
 \sum _{k=0}^{n-1-r} \mathds{S}_{z=x, \tilde{t}=0}(H^k)|_{p_i}\mathds{S}_{z=y, \tilde{t}=0}(H^{n-1-r-k})|_{p_i} 
 \right. 
  \\  \left.
 + \frac{1}{\lambda_0^n} \sum _{b=0}^{r-1}
\mathds{S}_{z=x, \tilde{t}=0}(H^{n-r+b})|_{p_i}\mathds{S}_{z=y, \tilde{t}=0}(H^{n-1-b})|_{p_i} \right)      . \end{multline}

\subsection{Vertex terms}

Applying \eqref{asymJ} to 
$$\mathds{I} |_{\tilde{t}=0, p_i} = I_{\T} |_{t=0, p_i} \equiv
\sum _{d=0}^{\infty}  q^d \frac{\prod _{a=1}^r \prod _{k=1}^{l_ad}  (l_a + k\frac{z}{\lambda _i})}{\prod _{k=1}^d ((1+k\frac{z}{\lambda _i})^n - 1)}
 \equiv I_{\T}|_{t=0, p_n, z\mapsto z/\lambda_i} , $$ we obtain 
\begin{align*}  
 \mathds{I} |_{\tilde{t}=0, p_i} 
\equiv   e^{\mu (q)\lambda _i /z}  (\sum _{k=0}^{\infty} R_{k}(q) (z /\lambda _i)^k )  \end{align*}
for some  $\mu (q) \in q\QQ [[q]] $ and  $R_{k}(q) \in \QQ [[q]]$. Hence
\begin{align*}   \mathds{I} |_{\tilde{t}=t_H\tilde{H}, p_i} & \equiv e^{\kl_it_H /z}  (I_{\T} |_{t=0, q\mapsto qe^{t_H}} ) \\
& \equiv  e^{\kl_it_H /z}  e^{\mu (qe^{t_H})\lambda _i /z}  (\sum _{k=0}^{\infty} R_{k}(qe^{t_H}) (z /\lambda _i)^k ) .   \end{align*}
Thus
\begin{align} \label{Umu} \mathds{U}_i |_{\tilde{t}=t_H  \tilde{H} } & \equiv  \lambda _i (t_H + \mu (qe^{t_H})) \text{ and }  \\
      \mathds{R}_{i, k} & \equiv  R_{k}(qe^{t_H})/ (\lambda _i)^k . \nonumber \end{align}

Since 
\begin{align*} r_{i, 0}|_{t=0} & = \mathds{R}_{i, 0} |_{\tilde{t}=0} , \\ 
u_i|_{t=0} &  = \mathds{U}_i|_{\tilde{t}=0} , \text { and } \\
c_i(\lambda) & =  (\sum_{j\ne i} \frac{1}{\lambda _j - \lambda _i} )+ \sum _a \frac{1}{l_a\lambda _i}, \end{align*}
we conclude that
\begin{align}\label{final_vertex1}  \sum _i \frac{\log D_i|_{t=0}}{24} & \equiv \frac{ -n \log  R_0(q)}{24} , \\ 
\label{final_vertex2}
                \sum _i  \frac{c_i(\lambda) u_i |_{t=0}}{24} &\equiv \frac{1}{24} (\sum_a  \frac{n}{l_a} -\binom{n}{2}) \mu (q) . \end{align}

%%%%
%%%%

\subsection{Loop terms} 
If we let \begin{align*} 
\mathrm{W}_{p, p'}  & : =  (\left(\mathfrak{B}\right)^{p} \frac{\uuI _{\T} |_{t=0} }{I_0})|_{H= 1, z=x} 
(\left(\mathfrak{B}\right)^{p'}  \frac{\uuI_{\T}|_{t=0}}{I_0})|_{H= 1, z=y} \ \ ; \\
\mathrm{V}(x, y, q) & :=
  \sum _{p+p'=n-1-r}  \mathrm{W}_{p, p'}  + \sum _{p+p' = 2n -1 -r, \ n-r \le p\le n-1} \mathrm{W}_{p, p'}  , \end{align*}
then by \eqref{ExplicitS} and \eqref{EquivDelta},
\[ e_i\mathds{V}_{ii}(x,y) |_{\tilde{t}=0}\equiv 
\frac{\lambda _i ^{n-1-r}}{( \prod l_a )e_i  (x+y)}     \mathrm{V}(x/\lambda _i , y/\lambda _i,  q) . \]
Here we use also the degree property that $\deg (\mathds{S} (\phi _j) \mathds{S} (\phi ^j )) = n-1-r$.

Therefore
\begin{multline*} e^{-\mathds{U}_i |_{\tilde{t}=0}(1/x + 1/y)} e_i \mathds{V}_{ii}(x, y) |_{\tilde{t}=0}  \\ \equiv 
\frac{\lambda _i^{n-1}}{\ret (T_{p_i}Y) (x+y)}  e^{- \mu (q)(\lambda _i/x + \lambda _i/y)} \rV (x/\lambda _i, y/\lambda _i , q)  .
 \end{multline*}

Now the limit of $x, y \ra 0$ of \eqref{loop} (or equivalently the residue at $x=0$, $y=0$ of $\frac{\eqref{loop}}{xy}$) as computed  
in \cite[Lemma 5.4]{Popa1} becomes 
\begin{align*} \lim_{(x, y)\ra (0, 0)} (e^{-\mathds{U}_i(1/x + 1/y)}  e_i \mathds{V}_{ii}(x, y)  - \frac{1}{x+y} )|_{\tilde{t}=0}
\equiv  \frac{\lambda _i ^{n-2} }{\ret (T_{p_i} Y ) L(q)} q\frac{d}{dq} \mathrm{Loop} (q) ,\end{align*} 
where
\begin{align*} 
& L (q)  := (1-q\prod_a l_a^{l_a} )^{-1/n} \text { and } \\
& \mathrm{Loop} (q) := 
 \frac{n}{24} ( n-1 - 2 \sum_{a=1}^r \frac{1}{l_a}) \mu (q)    
   \\ & - \frac{3(n-1-r)^2 + (n-2)}{24}\log (1-q\prod l_a ^{l_a})  
       - \sum _{k=0}^{n-2-r}\binom{n-r-k}{2} \log C_k (q)  . \end{align*}

Since $$ \frac{\partial \mathds{U}_i}{\partial t_\gamma} |_{\tilde{t}=0}  = L(q) \lambda _i$$  by \eqref{Umu} and Proposition \ref{ExplicitProp} below,
we conclude that
\begin{align} \label{final_loop}     \sum _i  \mathbf{Loop}_i = \frac{1}{2} q\frac{d}{dq} \mathrm{Loop}(q).
\end{align}

\subsection{Proof of Theorem \ref{Main}}
Now the sum \[ \eqref{final_vertex1} + \eqref{final_vertex2} +  \frac{1}{2} \mathrm{Loop}(q) \]  can be explicitly obtained 
 by Proposition \ref{ExplicitProp} and hence we complete the proof of Theorem \ref{Main}.

\subsection{Explicit computations of $\mu, R_0, R_1$ and loop terms}
Recall we assume \eqref{roots}. Let $\lambda _0= 1$.
Note that 
$I_{\T}$ satisfies the differential equation (see \cite[Corollary 11.7]{Gequiv})
\[ \mathrm{PF} I_{\T} |_{t=t_H \cdot H} =0, \ \ \mathrm{PF}:= ( z\frac{d}{dt} )^n - 1 - 
q\prod _{a} \prod _{m=1}^{l_a} (l_az\frac{d}{dt} + mz) . \]
Applying the differential operator PF to the asymptotic form of 
\[ I_{\T} |_{t=t_H H, p_n}  \]  one 
obtain $\mu, R_0, R_1$ and the loop limit (see \S 4.2, \S 4.3 of \cite{Popa1} for details).

For the reader's convenience, we state the following Proposition due to Popa \cite{Popa1}. 

\begin{Prop}\label{ExplicitProp} \cite[Proposition 4.3 \& 4.4]{Popa1}

Consider $C_b$, $b=0, 1, ..., n-1$. 

\begin{enumerate}

\item\label{product} $\prod _{i=0}^{n-r} C_i = (1- q\prod_a l_a^{l_a} )^{-1}$.

\item $C_b = C_{n-r-b}$ for $b=0, 1, ..., n-r$.

\item\label{one} $C_b = 1$ for $b=n-r+1, ..., n-1$.

\item\label{muR}  \begin{align}\label{emu} \mu (q) = \int_0^q \frac{ (1-x\prod_a l_a^{l_a} )))^{-\frac{1}{n}}- 1}{x} dx ,  \end{align}
 \begin{align}\label{r0} R_0= L^{\frac{r+1}{2}} . \end{align}

\end{enumerate}
\end{Prop}

\begin{Rmk}\label{C_b} Let $\{(H^b)^\vee\}_b$ be the $E$-twisted Poincar\'e dual basis of $\{ H^b \} _b$. 
Note that \begin{align*} C_b & \equiv \lan (H^b)^\vee, H^{b-1} ; H \ran ^{0+, 0+}_{0, 2|1} , \text{  for } b=1, ..., n-1 \end{align*}
and hence item (2) of Proposition \ref{ExplicitProp} naturally follows except the claim $C_0 = C_{n-r}$.
\end{Rmk}

\end{document}